\documentclass[12pt,reqno,a4paper]{article}

\usepackage{amsmath,amsthm,amssymb}
\usepackage{authblk}
\usepackage{color}
\usepackage{xcolor}

\usepackage{enumitem}

\usepackage[T1]{fontenc}
\usepackage{graphics}
\usepackage{fullpage}
\usepackage[colorlinks,linkcolor=blue,backref=page]{hyperref}

\usepackage[normalem]{ulem}

\renewcommand*{\backrefalt}[4]{%
        \ifcase #1 (Not cited.)%
        \or        (Cited on page~#2.)%
        \else      (Cited on pages~#2.)%
        \fi}

\usepackage{mathtools}
\usepackage{multicol}
\usepackage{skull}

\usepackage{tikz}
\usetikzlibrary{cd}
\usetikzlibrary{decorations.markings}
\usepackage{circuitikz,wrapfig}

\usepackage[colorinlistoftodos,prependcaption]{todonotes}
\usepackage[all]{xy}

\makeatletter
\DeclareRobustCommand{\em}{%
  \@nomath\em \if b\expandafter\@car\f@series\@nil
  \normalfont \else \bfseries\itshape \fi}
\makeatother

\newtheorem{theorem}{Theorem}[section]

\newtheorem{lemma}[theorem]{Lemma}
\newtheorem{proposition}[theorem]{Proposition}

\theoremstyle{definition}
\newtheorem{definition}[theorem]{Definition}

\newtheorem{example}[theorem]{Example}
\theoremstyle{remark}
\newtheorem{remark}[theorem]{Remark}

\numberwithin{equation}{section}
\numberwithin{figure}{section}

\newcommand{\B}[1]{{\mathbf #1}}

\newcommand{\OP}{\operatorname}

\newcommand{\AL}[1]{\textcolor{red}{AL: #1}}

\begin{document}

\title{Are free groups of different ranks bi-invariantly quasi-isometric?}
\author{Jarek K\k{e}dra}
\affil{
University of Aberdeen and University of Szczecin\\
{\tt kedra@abdn.ac.uk}
}
\author{Assaf Libman}
\affil{
University of Aberdeen\\
{\tt a.libman@abdn.ac.uk}
}

\maketitle


\renewcommand{\thefootnote}{}




\renewcommand{\thefootnote}{\arabic{footnote}}
\setcounter{footnote}{0}


\begin{abstract}
We prove that a homomorphism between free groups of finite rank
equipped with the bi-invariant word metrics associated with
finite  generating sets is a quasi-isometry
if and only if it is an isomorphism.
\end{abstract}

\section{Introduction}\label{S:intro}

Let $\B F_m$ and $\B F_n$ be free groups of rank $m$ and $n$, respectively,
equipped with the bi-invariant word metrics associated with finite
generating sets (see Section \ref{S:preliminary} for definitions).

\begin{theorem}\label{T:main}
Let $n\geq 2$.
A homomorphism $\varphi\colon \B F_m\to \B F_n$ 
is a quasi-isometry if and only if it is an isomorphism.
\end{theorem}

It is well known that if both $\B F_m$ and $\B F_n$ are equipped with the
standard left-invariant word metrics associated with finite generating sets (such metrics are
not bi-invariant) then the inclusion of a finite index subgroup is a
quasi-isometry \cite[Corollary 5.31]{zbMATH06866242}. Thus the above theorem is
in contrast with the classical geometric group theory. It can be considered
as a form of rigidity where being a homomorphism and a quasi-isometry implies
being an isomorphism. It is interesting to what extent this rigidity can be generalised.
For example, we do not know whether a general map, not necessarily a
homomorphism, between free groups of different ranks can be a quasi-isometry.
In particular, the question in the title remains open (see also
\cite[Question 10.2.4]{zbMATH07779808} for a related question).
On the other hand, we know
that it does not hold for other groups. For example, the abelianisation
homomorphism $\B N\to \B Z^n$, from a torsion free nilpotent group is a quasi-isometry
\cite[Theorem 5.8]{zbMATH06532523}.
Bi-invariant word metrics and their 
general properties, particularly in connection to free groups, are discussed
in~\cite{zbMATH06532523,zbMATH06563655,zbMATH07843754}.

The proof of Theorem \ref{T:main} splits into three cases:
\begin{itemize}
\item[(a)]
If $\varphi$ is an isomorphism then it is a quasi-isometry.
\item[(b)] 
The image of $\varphi$ is a proper subgroup of finite index in $\B F_m$, 
in which case we prove that $\varphi$ is distorted by showing that
it carries an unbounded set in $\B F_m$ to a bounded set in $\B F_n$ (Proposition \ref{P:distorted}).
\item[(c)]
The image of $\varphi$ has  infinite index in $\B F_m$, in which case we show that $\varphi$ is not quasi-surjective (Proposition~\ref{P:not-qsur}).
\end{itemize}

\paragraph{Acknowledgements.} The question whether free groups of different
ranks equipped with bi-invariant word metrics are quasi-isometric was asked by
Henry Jaspars. We thank him and M. Brandenbursky for discussions. We also
thank Arielle Leitner, Federico Vigolo and Michał Marcinkowski for comments
on the initial version of this paper.

\section{Preliminaries}\label{S:preliminary}
The material presented in this section is well known and can be
found, for example, in Bridson-Haefliger \cite{zbMATH01385418} 
or Calegari \cite{zbMATH05577332}.

\subsection{Word metrics}\label{subsec:word metrics}

Let $G$ be a group generated by a set $S\subseteq G$. The corresponding
{\em word norm} is defined by
$$
|g|_S = \min\left\{n\in \B N\ |\ g = s_1\dots s_n,\ s_i\in S^{\pm 1}\right\}
$$
and the associated metric by $d_S(g,h) = |g^{-1}h|_S$. 
This metric is left-invariant.

The closure of $S \subseteq G$ to inverses and conjugation is the set
\begin{equation}\label{eqn:barS}
\overline{S} = \{ gsg^{-1} : s^{\pm 1} \in S, g \in G\}.
\end{equation}
We say that $S$ {\em normally generates} $G$ if $\overline{S}$ generates $G$.
In this case the word norm $|g|_{\overline{S}}$ on $G$ is invariant with respect 
to conjugation, we denote it by
\[
\| g \|_S 
\]
and call it the {\em conjugation-invariant word norm} on $G$ associated with $S$.
The metric $d_{\overline{S}}(x,y)=\|x^{-1}y\|_S=\|xy^{-1}\|_S$ on $G$ is bi-invariant.
We abusively denote it by $d_S$ (since we are only interested in bi-invariant metrics).
If $S$ is finite then the Lipschitz class of $d_S$ is maximal 
in the sense that for any bi-invariant metric $d$ on $G$, the identity map $\OP{id} \colon (G,d_S) \to (G,d)$ is Lipschitz.

\begin{lemma}\label{L:id-lip}
Let $S\subseteq G$ be a finite set normally generating $G$. Let $d$ be any
bi-invariant metric on $G$. Then the identity $\OP{Id}\colon (G,d_S)\to (G,d)$
is Lipschitz. In particular, bi-invariant word metrics associated with
finite normally generating sets are bi-Lipschitz equivalent.
\end{lemma}
\begin{proof}
Let $C = \max\{d(s,1)\ |\ s\in S\} = \max\{d(s,1)\ |\ s\in \overline{S}\}$, where
the second equality follows from the bi-invariance of $d$. 
Let $g,h\in G$ and let $n=d_S(g,h) = \|g^{-1}h\|_S$.
This means that $g^{-1}h=s_1\dots s_n$ for some $s_i\in \overline{S}$. We have the
following estimate, which shows that the identity is Lipschitz.
\begin{align*}
d(g,h) 
&= d(1,g^{-1}h)\\
&= d(1,s_1\dots s_n)\\
&\leq d(1,s_1)+ d(s_1,s_1\dots s_n) & \text{( by triangle inequality)}\\
&= d(1,s_1) + d(1,s_2\dots s_n) &\text{(by left-invariance)}\\
&\leq \sum_{i=1}^n d(s_i,1) \leq C\ n = C\ d(g,h).
\end{align*}
\end{proof}

Let $\B F_n = \langle s_1,\ldots, s_n\rangle$ be the free group of rank $n$.
We call the set $\{s_1,\ldots,s_n\}$ the {\em standard generating set}.
In this paper we are concerned with the bi-invariant word metrics
on free groups of finite rank associated with their standard generating sets.

\subsection{Quasi-isometries}

A map $\varphi\colon (X_1,d_1)\to (X_2,d_2)$ between metric spaces is called
a {\em quasi-isometry} if  
\begin{enumerate}
\item\label{item:qie}
There exist  $C>0$ and $D\geq 0$ such that for every $x,y\in X$,
\begin{equation}
\frac{1}{C}d_1(x,y)-D\leq d_2(\varphi(x),\varphi(y)) \leq Cd_1(x,y) + D
\label{Eq:qiso}
\end{equation}
\item\label{item:qs}
There exists $B\geq 0$ such that for every $y\in X_2$ there exists
$x\in X_1$ such that
\begin{equation}
d_2(\varphi(x),y)\leq B.
\label{Eq:qsur}
\end{equation}
\end{enumerate}
If $\varphi$ satisfies \eqref{Eq:qiso}
then it is called a {\em quasi-isometric embedding}; if it satisfies
the second inequality of \eqref{Eq:qiso} then it is called {\em large-scale Lipschitz}. 
If $\varphi$ satisfies \eqref{Eq:qsur} then it is called {\em quasi-surjective}.
If $\varphi$ is not a quasi-isometric embedding then it is called {\em distorted}.

\subsection{Quasi-morphisms}

A function $\psi\colon G\to \B R$ is called a {\em quasi-morphism} if there
exists $D\geq 0$ such that
\begin{equation}
|\psi(g) - \psi(gh) + \psi(h)|\leq D
\label{Eq:qmor}
\end{equation}
for all $g,h\in G$. The smallest of such $D$ is called the defect of $\psi$ and
denoted by $D_{\psi}$.

A quasi-morphism is  {\em homogeneous} if $\psi(g^n)=n\psi(g)$ for all
$g\in G$ and $n\in \B Z$. 
Every quasimorphism $\psi$ has an associated homogeneous quasimorphism defined by
$$
\overline{\psi}(g) = \lim_{n\to \infty} \frac{\psi(g^n)}{n}.
$$  
It is called the {\em homogenisation} of $\psi$, see \cite[Lemma 2.21]{zbMATH05577332}.
It is straightforward to check that if $\psi$ is homogeneous then $\overline{\psi}=\psi$.
From this it is easy to deduce that any homogeneous quasimorphism is a class function, i.e., it is constant on conjugacy classes in $G$, see \cite[Subsection 2.2.3]{zbMATH05577332}.

\begin{lemma}\label{L:qmor-lip}
Let $G$ be equipped with the bi-invariant word metric associated
with a normally generating subset $S\subseteq G$.
If $\psi\colon G\to \B R$ is a homogeneous quasi-morphism
bounded on $S$ then it is a Lipschitz function.
\end{lemma}

\begin{proof}
Suppose that $|\psi(s)|\leq B$ for some $B\geq 0$ and all $s\in S$.
Consider some $g \in G$ and set $\|g\|_S=n$.
By definition of the bi-invariant word metric, $g=s_1\dots s_n$ for some 
$s_i\in \overline{S}$. 
Let $D \geq 0$ be the defect of $\psi$.
Since $\psi$ is a class function $|\psi(s)| \leq B$ for all $s \in \bar{S}$ (see~\eqref{eqn:barS}).
Then
\begin{align*}
|\psi(g)| 
& = |\psi(s_1\dots s_n)|\\
&\leq \sum_{i=1}^n |\psi(s_i)| + (n-1)D  &\text{(see \cite[Lemma 2.17]{zbMATH05577332})} \\
& \leq (B+D)n
\\
&= (B+D)\|g\|.
\end{align*}
Let $g\neq h\in G$. 
Since $\psi$ is homogeneous and since $\| gh^{-1}\|$ is a positive integer 
\begin{align*}
|\psi(g)-\psi(h)|
&=|\psi(g) - \psi(gh^{-1}) + \psi(h^{-1}) + \psi(gh^{-1})| &
\\
&\leq D + |\psi(gh^{-1})| &
\\
&\leq D + (B+D) \cdot \|gh^{-1}\|   & \text{(by the calculation above)}\\
& \leq (B+2D) \cdot \|gh^{-1}\|   & \text{ (since $\| gh^{-1}\|$ is a positive integer)} \\
&= (B+2D) \cdot d(g,h). &
\end{align*}
\end{proof}

\subsection{The little counting quasi-morphism}\label{subsec:little counting}

Consider the free group $\B F_n$ 
with its the standard set of generators $S=\{s_1,\dots,s_n\}$.
In what follows the standard word norm $|g|_S$ will be denoted
by $\ell(g)$ to make various formulas easy to read.

\begin{definition}\label{def:small counting}
Let $1 \neq w \in \B F_n$ be presented as a reduced word.
The {\em little counting function} $c_w \colon \B F_n \to \B N$ is defined by
\[
c_w(g) = \max\{k : \text{the reduced form of $g$ contains $k$ {disjoint} copies of $w$} \}
\]
The {\em little counting quasi-morphism}  $\psi_w \colon \B F_n \to \B Z$  is the function
\[
\psi_w(g) = c_w(g)-c_{w^{-1}}(g).
\]
\end{definition}
\noindent
It is shown in \cite[Section 2.3.2]{zbMATH05577332} that $\psi_w$ is indeed a quasi-morphism with 
defect at most~$2$.

Let $g \in \B F_n$.
If $w \neq 1$, it is clear that $c_{w^{\pm 1}}(g) \leq \lfloor \ell(g)/\ell(w) \rfloor$ and therefore
\[
|\psi_w(g)| \leq \lfloor \ell(g)/\ell(w) \rfloor.
\]
Moreover, $\ell(g^k) \leq k\cdot \ell(g)$ and therefore $|\psi_w(g^k)| \leq \left\lfloor \frac{k \cdot \ell(g)}{\ell(w)} \right\rfloor$.
Since $\frac{1}{k} \cdot \left\lfloor \frac{k \cdot \ell(g)}{\ell(w)} \right\rfloor \xrightarrow{k \to \infty}\frac{\ell(g)}{\ell(w)}$, 
the homogenisation of $\psi_w$ satisfies
\begin{equation}\label{eq:little ell(g)/ell(w)}
|\overline{\psi_w}(g)| 
\leq
\frac{\ell(g)}{\ell(w)}.
\end{equation}
By \cite[Lemma 2.27]{zbMATH05577332} copies of $w$ in $g$ are disjoint from those of $w^{-1}$.
Assume that $1~\neq~w~\in~\B F_n$ is cyclically reduced.
Then for any $k>0$ the word representing $w^k$ is the concatenation of $k$ copies of $w$ which contains $k$ disjoint copies of $w$ leaving no room for copies of $w^{-1}$.
It follows that $\psi_k(w^k)=c_w(w^k)-c_{w^{-1}}(w^k)=k$.
We deduce that if $w$ is cyclically reduced then 
\begin{equation}\label{eq:psi_w(w)=1}
\overline{\psi_w}(w)=1.
\end{equation}

\section{Proof of the main result}
Throughout, free groups are equipped with the conjugation invariant word norm associated 
with their standard generating sets.
All groups are assumed to be equipped with conjugation-invariant norms.

The next two lemmas are the core of the proof of Theorem \ref{T:main}.
Given a homomorphism $i \colon \B F_m \to \B F_n$, if its image is a proper subgroup of 
finite index we will prove the existence of a  quasi-morphisms which satisfies the conditions 
of the first lemma, and if the index is infinite we will prove the existence of a quasi-morphism 
satisfying the conditions of the second lemma.

A function $f \colon G \to \B R$ is called homogeneous if $f(g^n)=n \cdot f(g)$ for all $n>0$.

\begin{lemma}\label{lem:not q-embedding}
Let $\rho \colon H \to G$ be a homomorphism.
Suppose that there exists a Lipschitz homogeneous function $ \phi \colon H \to \B R$ and elements $h_1,h_2 \in H$ such that
\begin{enumerate}[label=(\roman*)]
\item\label{item:not q-emb separation}
$\phi(h_1) \neq \phi(h_2)$.
\item\label{item:not q-emb conjugate}
$\rho(h_1), \rho(h_2)$ are conjugate in $G$.
\end{enumerate}
Then $\rho$ is not a quasi-isometric embedding.
\end{lemma}

\begin{proof}
We construct an unbounded $Y \subseteq H$ such that $\rho(H)$ is bounded in $G$.
Set 
\[
Y=\{h_1^k h_2^{-k} : k >0\}.
\]
Denote $g_1=\rho(h_1)$ and $g_2=\rho(h_2)$.
By assumption there exists $g \in G$ such that $g_2=gg_1g^{-1}$.
Then $\rho(Y)$ is bounded in $G$ since for any $k>0$
\[
\| \rho(h_1^k h_2^{-k})\|_G = \| g_1{}^k \cdot g_2{}^{-k} \|_G =\| g_1{}^k g g_1{}^{-k} g^{-1}\|_G \leq \| g_1{}^k g g_1{}^{-k}\|_G+\| g^{-1}\|_G \leq 2\|g\|.
\]
Next, let $C>0$ be the Lipschitz constant of $\phi$.
Then for any $k>0$
\[
C \cdot \|h_1^k h_2^{-k}\|_H 
=
C \cdot d_H(h_1^k,h_2^{k}) 
\geq 
| \phi(h_1^k)-\phi(h_2^k)| 
= 
k \cdot |\phi(h_1)-\phi(h_2)|.
\]
Since $\phi(h_1) \neq \phi(h_2)$ it follows that $\| h_1^k h_2^{-k}\|_H \xrightarrow{k \to \infty} \infty$, hence $Y$ is unbounded.
\end{proof}

\begin{lemma}\label{lem:not q-sur}
Let $\rho \colon H \to G$ be a homomorphism.
Assume that there exists a Lipschitz homogeneous function $\phi \colon G \to \B R$ such that 
\begin{enumerate}[label=(\roman*)]
\item
$\phi \neq 0$.
\item
$\phi \circ \rho =0$.
\end{enumerate}
Then $\rho$ is not quasi-surjective.
\end{lemma}

\begin{proof}
By assumption there exists $g \in G \setminus \rho(H)$ such that $\phi(g) \neq 0$.
Let $C>0$ be the Lipschitz constant of $\phi$.
For any $h \in H$ and any $n >0$
\[
C \cdot d_G(g^n,\rho(h)) 
\geq
|\phi(g^n) - \phi(\rho(h))|
= 
|\phi(g^n)|
= 
n \cdot |\phi(g)|.
\]
It follows that $d(g^n, \rho(H)) \geq C^{-1}|\phi(g)| \cdot n \xrightarrow{n \to \infty} \infty$.
This completes the proof.
\end{proof}

We now specialise to the case when $H=\B F_m$ and $G=\B F_n$ and prove Theorem \ref{T:main}.
The maps $\phi$ in Lemmas \ref{lem:not q-embedding} and \ref{lem:not q-sur} will be obtained as homogeneous quasimorphisms.
The Lipschitz condition in these Lemmas is a consequence of Lemma \ref{L:qmor-lip}.
To achieve condition \ref{item:not q-emb separation} in Lemma \ref{lem:not q-embedding} we need the next lemma.

\begin{lemma}[Separation Lemma]\label{L:separation}
Let $g,h\in \B F_n$ be such that the subgroups $\langle g \rangle$ and $\langle h \rangle$ are not conjugate.
Then there exists a homogeneous quasi-morphism 
$\psi\colon \B F_n \to \B R$ such that $\psi(g)\neq\psi(h)$.
\end{lemma}
\begin{proof}
Since homogeneous quasi-morphisms are class functions, we can assume without losing generality that $g$ and $h$ are cyclically reduced and that $\ell(h)\leq \ell(g)$.
Then $g \neq 1$ since otherwise also $h=1$ which contradicts the hypothesis. 

Let $\psi_{g}\colon \B F_n\to \B Z$ be the little counting quasi-morphism, see Definition \ref{def:small counting}, and let $\overline{\psi}_{g}$ denote its homogenisation. 
Since $g$ is cyclically reduced, it follows from \eqref{eq:psi_w(w)=1} that
\begin{equation}\label{eqn:psi=1}
\overline{\psi}_g(g) = 1.
\end{equation}
We complete the proof by showing that $|\overline{\psi}_{g}(h)|< 1$.
If $\ell(h)<\ell(g)$ then by \eqref{eq:little ell(g)/ell(w)}
\[
\left|\overline{\psi}_{g}(h)\right| \leq \ell(h) /\ell(g) < 1
\]
and we are done.
If $\ell(h)=\ell(g)$ we claim that $\psi_g(h^k)=0$ for all $k \geq 0$.
Otherwise, if $\psi_g(h^k) \neq 0$ for some $k$ then either $g$ or $g^{-1}$ must be a subword of the cyclically reduced word $h^k$.
Since $\ell(h)=\ell(g)$ this implies that the reduced word representing $g^{\pm 1}$ is a cyclic permutation of $h$, i.e., $h$ is conjugate to $g^{\pm 1}$.
In particular $\langle g \rangle$ is conjugate to $\langle h \rangle$ which contradicts the hypothesis.
It follows that $\overline{\psi_g}(h)=0$ and this completes the proof.
\end{proof}

We recall that centralisers in the free group are cyclic, generated by elements
which are not proper powers.  Also, no element of the free group other than the
identity is conjugate to its inverse \cite[Proposition 2.19]{zbMATH01554175}.

\begin{proposition}\label{P:distorted}
Let $n \geq 2$.
Any homomorphism $\rho \colon \B F_m \to \B F_n$ whose image is a proper 
subgroup of finite index in $\B F_n$ is not a quasi-isometric embedding.
\end{proposition}

\begin{proof}
Set $H=\B F_m$ and $G=\B F_n$.
We prove the existence of $\phi$ and $h_1, h_2$ in Lemma \ref{lem:not q-embedding}.

Since $[G:\rho(H)]<\infty$ there exists $N \leq \rho(H)$ such that 
$N$ is normal and of finite index in $G$ \cite[Corollary 4.16]{zbMATH00870511}.
Let $x,y$ be distinct generators of $\B F_n$.
Since $G/N$ is finite, there exist $k, \ell>0$ such that $x^k, y^\ell \in N$.
Set $g_1=x^k y^\ell$.
Since $g_1$ is not a proper power in $\B F_n$, the centraliser $\OP{C}_{\B F_n}(g_1)$
of $g_1$ in $\B F_n$ is the cyclic subgroup generated by $g_1$:
\[
\OP{C}_{\B F_n}(g_1) = \langle g_1 \rangle  \leq N \leq \rho(H).
\]
Choose some $u \in G \setminus \rho(H)$ and set $g_2 = ug_1u^{-1}$.
By construction, $g_1$ and $g_2$ are conjugate in $G$ and $g_2 \in N$ since $N$ is normal.
We claim that the subgroups $\langle g_1 \rangle$ and $\langle g_2 \rangle$ are not conjugate in $\rho(H)$.
If they are then there exists $v \in \rho(H)$ such that 
\[
\langle g_2 \rangle = \langle vg_1v^{-1} \rangle.
\]
Since $\langle g_1 \rangle$ and $\langle g_2 \rangle$ are infinite cyclic, either $g_2=vg_1v^{-1}$ or $g_2=vg_1^{-1}v^{-1}$.
If $g_2=vg_1v^{-1}$ then $vg_1v^{-1}=ug_1u^{-1}$ which implies that 
$v^{-1}u \in \OP{C}_{\B F_n}(g_1) \leq \rho(H)$, so $u \in v \cdot \rho(H)=\rho(H)$, contradicting the choice of $u$.
If $g_2=vg_1^{-1}v^{-1}$ then $vg_1^{-1}v^{-1}=ug_1u^{-1}$ which implies that $g_1^{-1}$ is conjugate in $\B F_n$ to $g_1$, which is again a contradiction since $g_1 \neq 1$.
It follows that $\langle g_1 \rangle$ and $\langle g_2 \rangle$ are not conjugate in $\rho(H)$ as claimed.

Let $h_1,h_2 \in H$ be in the preimages of $g_1,g_2$, respectively.
Then $\langle h_1 \rangle$ and $\langle h_2 \rangle$  cannot be conjugate in $H$. 
By the Separation Lemma \ref{L:separation} there exists a homogeneous quasi-morphism $\psi \colon H \to \B R$ such that $\psi(h_1) \neq \psi(h_2)$.
By Lemma \ref{L:qmor-lip} $\psi$ is Lipschitz.
Then $\rho$ is not a quasi-isometric embedding by Lemma \ref{lem:not q-embedding}
\end{proof}

We now turn to deal with homomorphisms $\B F_m \to \B F_n$ whose images have infinite index.
We need a machinery to construct maps $\phi$ as in Lemma \ref{lem:not q-sur}.

A {\em killer word} for a subgroup $G \leq \B F_n$ is a reduced word $w \in \B F_n$ which is not a subword in any $g \in G$. 
The reason for this terminology is that the little counting quasi-morphism $\psi_w$ vanishes on $G$.
In fact, $\psi_u|_G=0$ for any word $u$ which contains $w$.

\begin{lemma}\label{lem:Pagliantini Rolli}
Let $n \geq 2$.
Any  $G \leq \B F_n$ of infinite index and finite rank admits a killer word.
\end{lemma}

\begin{proof}
A direct consequence of \cite[Lemmas 4.1 and 4.6]{zbMATH06424932}.
\end{proof}

Pagliantini and Rolli's argument in \cite{zbMATH06424932} is indirect.
In the Appendix we give a constructive proof of Lemma \ref{lem:Pagliantini Rolli} which gives an effective way to generate killer words.

\begin{proposition}\label{P:not-qsur}
Let $n \geq 2$.
Any homomorphism $\rho\colon \B F_m\to \B F_n$ with image of infinite index is not quasi-surjective.
\end{proposition}

\begin{proof}
Set $H=\B F_m$ and $G=\B F_n$.
By Lemma \ref{lem:Pagliantini Rolli} there exists a reduced word $w' \in \B F_n$ which is not a subword in any $u \in \rho(H)$.
Since $n \geq 2$ we may multiply $w'$ on the left or on the right by an appropriate generator of $\B F_n$ to obtain a cyclically reduce $w$ which is not a subword in any $u \in \rho(H)$.
Clearly, the same is then true for $w^{-1}$.
As a result, the little counting quasi-morphism $\psi_w \colon G \to \B R$ vanishes on $\rho(H)$.
Consequently, the same holds for its homogenisation $\overline{\psi_w}$, which is non-trivial by \eqref{eqn:psi=1}.
By Lemma \ref{L:qmor-lip} $\overline{\psi}_w$ is Lipschitz and by construction $\overline{\psi}_w \circ \rho=0$.
The result follows from Lemma \ref{lem:not q-sur}.
\end{proof}

\begin{remark}\label{E:dihedral}
Consider $\B F_2=\B Z*\B Z$ with the canonical set  $S=\{a,b\}$ of free generators.
Let $\pi \colon \B Z*\B Z \to \B Z_2 * \B Z_2$ be the canonical quotient and $T=\{\bar{a},\bar{b}\}$ the image of $S$.
Let $G$ be the kernel of $\pi$.
Then $G$ is a free group, since it is a subgroup of a free group, of infinite rank.
One easily checks that $\OP{diam}(\B Z_2*\B Z_2,\| \, \|_T)=2$ and that any $x \in \B Z_2*\B Z_2$ lifts to $\tilde{x} \in \B F_2$ such that $\|x\|_T=\|\tilde{x}\|_S$.
From this it follows that $d(y,G) \leq 2$ for any $y \in \B F_2$, in particular the inclusion $i \colon G \to \B F_2$ is quasi-surjective.
This shows that the assumption in Proposition~\ref{P:not-qsur} that the groups are of finite ranks is essential.

On the other hand, $G \lhd \B F_2$ and the elements $a^2b^2$ and 
$ab^2a = a^{-1}a^2b^2a$ are conjugate in $\B F_2$ but not in $G$.
They are also not proper powers.
By Lemmas \ref{L:separation}, \ref{L:qmor-lip} and \ref{lem:not q-embedding} the  inclusion $i \colon G \to \B F_2$ is not a quasi-isometric embedding where $G$ is equipped with the conjugation invariant word norm with respect to some (infinite) generating set.
\hfill $\diamondsuit$
\end{remark}


We remark that any  $1 \neq N \lhd \B F_n$ must be unbounded.
Indeed, $N$ must contain a cyclically reduced $w \neq 1$, then by \eqref{eq:psi_w(w)=1} $\overline{\psi_w}(w^k)=k$ 
, and since $\overline{\psi_w}$ is Lipschitz by Lemma \ref{L:qmor-lip}, $\|w^k\| \xrightarrow{k \to \infty} \infty$.

\begin{proof}[Proof of Theorem \ref{T:main}]
Let $\varphi\colon \B F_m\to \B F_n$ be a homomorphism. If it is an isomorphism 
then $\varphi$ is a quasi-isometry since it is (quasi) surjective  and since both $\varphi$ and $\varphi^{-1}$ are Lipschitz by Lemma \ref{L:id-lip}


Suppose that $\varphi$ is not an isomorphism.
If $\varphi$ is not injective then $\ker \varphi$ is an unbounded subset whose image is bounded, so $\varphi$ is not a quasi-isometric embedding.
If $\varphi$ is not surjective then it cannot be a quasi-isometry by  Propositions \ref{P:distorted} and \ref{P:not-qsur}.
\end{proof}

\appendix

\section{Killer words}

The valence of a vertex $v$ in a graph $\Gamma$ is $k$ if a neighbourhood of $v$ is homeomorphic to a point with $k$ whiskers.
The {\em initial} and {\em terminal} vertices of a simplicial path $\gamma$ in $\Gamma$ are denoted $\OP{ini}(\gamma)$ and $\OP{term}(\gamma)$.
A simplicial path in $\Gamma$ is called {\em reduced} if it has no backtracks.
A subpath $\gamma'\subseteq \gamma$ is called 
a {\em prefix} of $\gamma$ 
if $\gamma=\epsilon_1\dots \epsilon_n$ is the concatenation of simplicial edges and $\gamma'=\epsilon_1\dots\epsilon_k$ for some $0 \leq k \leq n$.
A simplicial path $\gamma$ is called {\em self intersecting} if the sequence of vertices $v_0,v_1,\dots,v_n$ it visits contains repetitions.
In particular, $\OP{length}(\gamma) \geq 1$.

\begin{lemma}\label{lem:continue path}
Let $\Gamma$ be a connected graph with finitely many vertices, possibly with
multiple edges and loops.
Suppose that the valence of every $v \in V(\Gamma)$ is at least $2$.
Let $\gamma$ be a reduced simplicial path in $\Gamma$.
Then every vertex of $\Gamma$ is the terminus of some reduced simplicial path $\tilde{\gamma}$ 
with prefix $\gamma$.
\end{lemma}

\begin{proof}
Let $u$ be the terminus of $\gamma$.
Let $\Pi$ be the set of all simplicial paths $\pi$ in $\Gamma$ starting at $u$ such that $\gamma \pi$ has no backtracks, thus, $\gamma \pi$ is reduced with prefix $\gamma$.
Set
\[
T=\{\OP{term}(\pi) : \pi \in \Pi\}.
\]
Since $\OP{term}(\gamma\pi)=\OP{term}(\pi)$, we will complete the proof by proving that $T=V(\Gamma)$.

Since the valence of every $v \in V(\Gamma)$ is at least $2$, it follows that $\Pi$ contains arbitrarily long paths.
Since $V(\Gamma)$ is finite, $\Pi$ must contain self-intersecting paths and we choose such a path $\pi$ of minimum length.
Say $\pi$ is the concatenation of simplicial edges $\epsilon_1 \dots \epsilon_n$ visiting the vertices $v_0,\dots,v_n$.
The minimality of $\OP{length}(\pi)$ implies that $v_n=v_k$ for some $0 \leq k<n$ (otherwise we can discard the last edge $\epsilon_n$ in $\pi$).
We now continue the path $\pi$ by backtracking along $\pi$ from $v_k$ down to $v_0=u$ and obtain in this way a loop $\lambda=\pi \overline{\epsilon_k} \dots \overline{\epsilon_1}$ from $u$ to $u$. 
The minimality of $\OP{length}(\pi)$ implies that $\lambda$ has no backtracks because a backtrack in $\lambda$, if it exists, can only occur at the juncture of $\pi$ with $\overline{\epsilon_k}$ which implies that $k \geq 1$ and that $v_{n-1}=v_{k-1}$ which contradicts the minimality of $\OP{length}(\pi)$.
Thus,
$
\lambda \in \Pi,
$
and $\OP{length}(\lambda) \geq 1$ (since it contains the self-intersecting $\pi$ as a prefix).
We are now ready to prove that $T=V(\Gamma)$.

First, $T \neq \emptyset$ since $\Pi$ contains the trivial path from $u$ to $u$.
Assume that $T \subsetneq V(\Gamma)$.
Choose $v \in V(\Gamma) \setminus T$ at distance $1$ from some $v' \in T$.
Let $\epsilon$ be a simplicial edge from $v'$ to $v$.
Since $v' \in T$, there exists $\pi \in \Pi$ such that $v'=\OP{term}(\pi)$.
If $v'=u$ we choose $\pi$ to be the loop $\lambda$ we constructed above.
In either case, whether $v'=u$ or not, $\OP{length}(\pi) \geq 1$.
Clearly, $\pi \epsilon$ has no backtracks because $\pi$ has this property and all the vertices of $\pi$ belong to $T$ while $\OP{term}(\epsilon) \notin T$.
Since $\OP{length}(\pi) \geq 1$ it is clear that $\gamma \pi \epsilon$ is reduced because $\gamma\pi$ is. 
We deduce that  $\pi \epsilon \in \Pi$.
But then $v=\OP{term}(\pi \epsilon) \in T$ which is a contradiction.
\end{proof}

Recall that a {\em killer word} for $G \leq \B F_n$ is a reduced word which is not a subword in any $g \in G$.

\begin{theorem}\label{T:killer}
Let $G\leq \B F_n$ be a finitely generated subgroup of infinite index. 
Then $G$ admits  a killer word $w \in \B F_n$.
\end{theorem}

\begin{proof}
Let $S=\{s_1,\ldots,s_n\}$ be the standard generating set of $\B F_n$.  Let
$X=\bigvee \B S^1$ be a graph with one vertex $x\in X$ and $n$ edges, 
all of which are loops, so that $\pi_1(X,x)\cong \B F_n$. 
The graph $X$ is directed and each edge is labelled
by a generator $s\in S$. Moreover, the valence of the vertex $x$ is equal to
$2n$.  Let $p\colon (\widetilde{X},\tilde{x})\to (X,x)$ be a covering projection
corresponding to the subgroup $G\leq \pi_1(X,x)$.  We equip $\widetilde{X}$ with the
directed labelled graph structure such that the projection $p$ is a morphism of
directed labelled graphs.

Since $G$ is finitely generated $\widetilde{X}$ contains a finite connected subgraph $\Gamma$ which contains $\tilde{x}$, has no vertices of valence $1$ and is a deformation retract of $\widetilde{X}$.
This is the {\it core} graph in the terminology of Stallings \cite{zbMATH03824074}.
The restriction $p\colon \Gamma\to X$
of the covering map is not a covering any more, since otherwise $G$ would be of
finite index. It implies that some vertices of $\Gamma$ have valence smaller
than $2n$.  We call them {\em bad}.

Recall that the edges of the graph $\widetilde{X}$ are oriented and labelled by the letters of the alphabet $S$.
Hence, the labels of the edges in a simplicial path $\pi$ in $\widetilde{X}$ determine 
a word $w$ in the alphabet $S^{\pm 1}$ 
which is reduced if and only if the path $\pi$ is reduced.
Conversely, any vertex $v \in V(\widetilde{X})$ and any word $w$ in the alphabet $S^{\pm 1}$ determine a unique simplicial path in $\widetilde{X}$  which we denote by
$\OP{path}(v,w)$.
To see how this path is constructed, observe that $\widetilde{X}$ is a covering of $X$ so
for any $v \in V(\widetilde{X})$ and any $s \in S$ there is a unique directed edge $e_s \in \Gamma$ with label $s$ emanating from $v$ and a unique directed edge $e_{s^{-1}} \in \Gamma$ with label $s$ terminating at $v$.
In the first case we obtain the simplicial edge $\epsilon_s=e_s$, and in the second the simplicial edge $\epsilon_{s^{-1}}=\overline{e_{s^{-1}}}$, both have initial vertex $v$.
If $w=w_1\dots w_n$ is a word in the alphabet $S^{\pm 1}$, then $\OP{path}(v,w)$ is the concatenation of the simplicial edges $\epsilon_{w_1}, \dots ,\epsilon_{w_n}$ described above, where $\epsilon_{w_1}$ starts at $v$ and for every $2 \leq i \leq n$ the simplicial edge $\epsilon_{w_i}$ starts at $\OP{term}(\epsilon_{w_{i-1}})$.

We are now ready to construct a killer word for $G$.
Before we start, recall that by construction every  $g\in G$ is represented by a reduced simplicial loop in $\Gamma$ based at $\tilde{x}$.  

Let $v_1,\ldots, v_m$ be an enumeration of the vertices of $\Gamma$.  
We construct by induction reduced words $w_0,\dots,w_m$, such that 
each is a prefix of its successor, with the property that for any 
$0 \leq k \leq m$, 
\[
1 \leq i \leq k \ \implies \ \OP{path}(v_i,w_k) \nsubseteq \Gamma.
\]
To start the induction, set $w_0$ to be the empty word.
The condition holds vacuously.
Assume that $w_{k-1}$ has been constructed for some $1 \leq k \leq m$.
Consider $\pi = \OP{path}(v_k,w_{k-1})$.
It is reduced since $w_{k-1}$ is.
If $\pi \nsubseteq \Gamma$ then set $w_k=w_{k-1}$.
Then $\OP{path}(v_k,w_k) \nsubseteq \Gamma$ by assumption and $\OP{path}(v_i,w_k) \nsubseteq \Gamma$ for all $1 \leq i \leq k-1$ by construction of $w_{k-1}$.
If $\pi \subseteq \Gamma$, use Lemma \ref{lem:continue path} to continue it to a reduced path $\pi' \subseteq \Gamma$ whose terminus is a bad vertex $u'$.
Continue $\pi'$ along a simplicial edge $\epsilon$ in $\widetilde{X}$ not in $\Gamma$.
Then $\pi' \epsilon$ is a reduced path and let $w_k$ be the word associated to it.
Now, $w_{k-1}$ is a prefix of $w_k$ since $\pi$ is a prefix of $\pi'$.
Therefore $\OP{path}(v_i,w_k) \nsubseteq \Gamma$ for all $1 \leq i \leq k-1$ by construction of $w_{k-1}$, and $\OP{path}(v_k,w_k)=\pi'\epsilon \nsubseteq \Gamma$ by construction.
This completes the induction step of the construction.

Set $w=w_m$.
By construction,
$
\OP{path}(v,w) \nsubseteq \Gamma,
$ 
for any $v \in V(\Gamma)$.
We will finish the proof by showing that $w$ is a killer word for $G$.

Consider an arbitrary $g \in G$ presented as a reduced word in $\B F_n$.
We claim that $w$ cannot be a subword of $g$.
Let $\gamma \subseteq \Gamma$ be a reduced simplicial loop based at $\widetilde{x}$ which represents $g$.
Let $u$ be the word in the alphabet $S^{\pm 1}$ that $\gamma$ determines.
Then $\gamma=\OP{path}(\widetilde{x},u)$ and since $\B F_n$ is free, $g=u$ as reduced words.
Suppose that  $w$ is a subword of $g$.
Then $u=u' w u''$ for some subwords $u',u''$ of $u$.
Let $v'$ be the terminus of $\OP{path}(\widetilde{x},u')$ and $v''$ the terminus of $\OP{path}(v',w)$.
Then $v' \in V(\Gamma)$ since $\OP{path}(\widetilde{x},u') \subseteq \gamma \subseteq \Gamma$, and 
\[
\gamma = \OP{path}(\widetilde{x},u) =  \OP{path}(\widetilde{x},u') \cdot \OP{path}(v',w) \cdot \OP{path}(v'',u'').
\]
In particular, $\OP{path}(v',w) \subseteq \Gamma$ which is a contradiction, 
since $v' \in V(\Gamma)$.
\end{proof}

The proof of Theorem \ref{T:killer} gives an algorithm to find killer words which we demonstrate in the next example.

\begin{example}\label{E:killer}
Let $G\leq \B F_2=\langle a,b\rangle$ be generated by $aba^{-1}b^{-1}$, $b^4$ and $a^3$.
The corresponding graph $\Gamma$ is pictured below. All vertices except $v_0$ are bad
and are marked red.

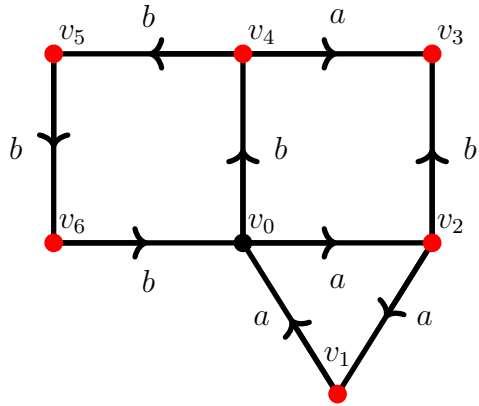
\begin{center}
\begin{tikzpicture}[line width=2pt, scale=0.5]
\draw (0,0) -- (5,0) -- (10,0) -- (10,5) -- (5,5) -- (0,5) -- (0,0);
\draw (5,0) -- (5,5);
\draw (5,0) -- (7.5,-4) -- (10,0);

\draw [->] (0,0) -- (2.5,0);
\draw [->] (5,0) -- (7.5,0);
\draw [->] (10,0) -- (10,2.5);
\draw [->] (5,5) -- (7.5,5);
\draw [->] (5,0) -- (5,2.5);
\draw [->] (5,5) -- (2.5,5);
\draw [->] (0,5) -- (0,2.5);
\draw [->] (10,0) -- (8.75,-2);
\draw [->] (7.5,-4) -- (6.25,-2);

\filldraw (5,0) circle [radius=5pt];
\filldraw [red] (0,0) circle [radius=5pt]
 (10,0) circle [radius=5pt]
 (10,5) circle [radius=5pt]
 (5,5) circle [radius=5pt]
 (0,5) circle [radius=5pt]
 (7.5,-4) circle [radius=5pt];

\draw (5.5,0.5) node {$v_0$};
\draw (10.5,0.5) node {$v_2$};
\draw (10.5,5.5) node {$v_3$};
\draw (5.5,5.5) node {$v_4$};
\draw (0.5,5.5) node {$v_5$};
\draw (7.5,-3) node {$v_1$};
\draw (0.5,0.5) node {$v_6$};

\draw (2.5,-1) node {$b$};
\draw (7.5,-1) node {$a$};
\draw (11,2.5) node {$b$};
\draw (7.5,6) node {$a$};
\draw (6,2.5) node {$b$};
\draw (-1,2.5) node {$b$};
\draw (2.5,6) node {$b$};
\draw (5.5,-2) node {$a$};
\draw (9.8,-2) node {$a$};
\end{tikzpicture}
\end{center}

The path from $v_0$ to $v_1$ defines an element $a^{-1}$. Augmenting it with $b$ yields
an element $a^{-1}b$ that cannot be obtained as a path starting at $v_0$. 
The path $v_1v_2v_3$ defines an element $a^{-1}b$ and augmenting it with $a$ yields
a word $a^{-1}ba$ that cannot be obtained as a path from either $v_0$ or $v_1$. It can, however,
be obtained as a path $v_2v_0v_4v_3$, so the new word is $a^{-1}ba^2$. It cannot be
obtained from a path starting at $v_3$, $v_4$, $v_5$ or $v_6$. 
So $a^{-1}ba$ is a killer word for $G$.

This word is not cyclically reduced, but multiplying it on the right by the generator~$b$ gives a cyclically reduced killer word $a^{-1}ba^2b$.
\hfill $\diamondsuit$
\end{example}

\bibliography{/home/kedra/sync/bibliography/bibliography}
\bibliographystyle{plain}

\end{document}